\documentclass[twoside,centertags]{amsart}

\usepackage[cmtip,color,all]{xy}
\usepackage{amsmath}
\usepackage{amsfonts}
\usepackage{amssymb}
\usepackage{amsthm}
\usepackage{graphicx}
\usepackage{mathrsfs}
\usepackage{hyperref}

\newtheorem{theorem}{Theorem}[section]
\newtheorem{corollary}[theorem]{Corollary}

\newtheorem{proposition}[theorem]{Proposition}

\theoremstyle{definition}
\newtheorem{definition}[theorem]{Definition}

\theoremstyle{definition}
\newtheorem{remark}[theorem]{Remark}
\newtheorem{example}[theorem]{Example}

\newtheorem{application}[theorem]{Application}

\newcommand{\D}{\mathcal{D}}
\newcommand{\Sp}{\mathcal{S}}

\newcommand{\op}{\mathrm{op}}

\newcommand{\Z}{\mathbb{Z}}
\newcommand{\U}{\mathcal{U}}
\newcommand{\V}{\mathcal{V}}
\newcommand{\I}{\mathcal{I}}
\newcommand{\R}{\mathbb{R}}
\newcommand{\HR}{\mathrm{H}\mathbb{R}}
\newcommand{\Top}{\mathrm{Top}}
\newcommand{\Ch}{\mathrm{Ch}}
\renewcommand{\lim}{\mathrm{lim}}
\newcommand{\colim}{\mathrm{colim}}
\newcommand{\hocolim}{\mathrm{hocolim}}

\newcommand{\ch}{\mathcal{D}}
\newcommand{\W}{\mathcal{W}}

\newcommand{\ImFU}{\mathcal{T}(\widetilde{F}_{\U})}

\newcommand{\G}{\widetilde{G}}
\newcommand{\pt}{\star}

\begin{document}

\title[Bounded cohomology and homotopy colimits]{Bounded cohomology and homotopy colimits} %

\author[G. Raptis]{George Raptis}
\address{\newline G. Raptis \newline
Fakult\"{a}t f\"ur Mathematik \\
Universit\"{a}t Regensburg \\
93040 Regensburg, Germany}
\email{georgios.raptis@ur.de}

\begin{abstract}
The comparison map from bounded cohomology to singular cohomology plays an important role in the study of bounded cohomology theory and its applications. The \emph{vanishing} and \emph{covering} theorems of Gromov and Ivanov show interesting and useful properties of the comparison map under appropriate assumptions. We discuss an approach to these theorems, based on the general homotopy-theoretic properties of the comparison map, and obtain new proofs and refined versions of these results.   
\end{abstract}

\maketitle

\setcounter{tocdepth}{1} 

\tableofcontents

\section{Introduction}

Bounded cohomology of topological spaces, introduced by Gromov \cite{Gr},  is a variant of singular cohomology which has deep connections with geometry 
and group theory. The bounded cohomology of a topological space $X$ (with coefficients in $\R$) is defined by restricting to those singular $\R$-cochains $\phi \in C^n(X; \R)$, $n \geq 0$, whose values $\{\phi(\sigma) \ | \ \Delta^n \xrightarrow{\sigma} X\}$ in $\R$ form a bounded subset.
This seemingly small modification turns out to have significant consequences and leads to a theory which differs substantially from 
singular cohomology. On the one hand, the singular $\R$-cochain complex
$$C^{\bullet}(-; \R) \colon \Top^{\op} \to \Ch_{\R}, \ X \mapsto C^{\bullet}(X; \R),$$
preserves weak equivalences and is excisive, i.e., the functor $C^{\bullet}(-;\R)$ sends weak homotopy equivalences to quasi-isomorphisms, and homotopy colimits of topological spaces to homotopy limits of cochain complexes. These fundamental properties (which are also true for general coefficients) allow us to determine and compute the singular cohomology of $X$ from the singular cohomologies of the spaces in a homotopy colimit decomposition of $X$. On the other hand, the functor of bounded  $\R$-cochains 
$$C^{\bullet}_b(-; \R) \colon \Top^{\op} \to \Ch_{\R}, \ X \mapsto C^{\bullet}_b(X; \R),$$
preserves weak equivalences \cite{Iv}, but it is not excisive. The failure to satisfy excision makes bounded cohomology difficult to compute in general. 

\smallskip

The inclusion of bounded $\R$-cochains into all singular $\R$-cochains defines a natural comparison map of cochain complexes
$$c_X \colon C^{\bullet}_b(X; \R) \rightarrow C^{\bullet}(X; \R)$$
and corresponding comparison maps after passing to cohomology
$$c_X^n \colon H^n_b(X; \R) \rightarrow H^n(X; \R).$$
The study of the properties of these comparison maps to singular cohomology is an important tool in bounded cohomology theory and its applications. Moreover, these comparison maps yield information about the simplicial volume \cite{Gr}: if $X$ is an oriented closed connected $n$-manifold, then the simplicial volume $||X||$ of $X$ vanishes if and only if $c^n_X$ is the zero map.  

Gromov \cite{Gr} and Ivanov \cite{Iv, Iv2} established interesting properties of the comparison maps $\{c^n_X\}_{n \geq 0}$ in the case where $X$ admits an open cover with suitable properties (see also \cite{FM, LS}). We recall that an open cover $\U = \{U_i\}_{i \in I}$ of a topological space $X$ is called \emph{amenable} if the image of the homomorphism $\pi_1(U_i, x) \to \pi_1(X, x)$ is amenable for every $i \in I$ and $x \in U_i$; we do not assume here that $U_i$ is path-connected. For any open cover $\U = \{U_i \}_{i \in I}$ of $X$, we denote by $\I_{\U}$ the poset of finite subsets $\sigma \subseteq I$ such that $U_{\sigma} : = \bigcap_{i \in \sigma} U_i \neq \varnothing$, ordered by reverse inclusion of subsets. Let $|\I_{\U}|$ denote the classifying space of $\I_{\U}$ -- this is the geometric realization of the nerve of $\I_{\U}$, that is, the simplicial complex associated to the poset of chains in $\I_{\U}$.

\begin{theorem}[Gromov \cite{Gr}, Ivanov \cite{Iv}] \label{nerve_theorem}
Let $X$ be a path-connected topological space and let $\U = \{U_i\}_{i \in I}$ be an amenable open cover of $X$. Then the following hold:
\begin{itemize}
\item[(1)](Covering theorem) Suppose that $U_{\sigma}$ is path-connected for every $\sigma \in \I_{\U}$. Then there are $\R$-linear maps 
$$\phi^n \colon H^n_b(X; \R) \rightarrow H^n(|\I_{\U}|; \R), \ \ n \geq 0,$$
such that the following diagrams commute:
$$
\xymatrix{
H^n_b(X; \R) \ar[r]^{c^n_X} \ar[rd]_{\phi^n} & H^n(X; \R). \\
& H^n(|\I_{\U}|; \R) \ar[u]_{H^n(p_{F_{\U}}; \R)} 
}
$$
\item[(2)](Vanishing theorem) If the dimension of $|\I_{\U}|$ is $\leq m-1$, then the comparison map $c^n_X$ vanishes for $n \geq m$. 
\end{itemize}
\end{theorem}

Theorem \ref{nerve_theorem}(2) is originally due to Gromov \cite{Gr}; several proofs of this result have appeared in the literature (see Ivanov\cite{Iv}, Frigerio--Moraschini \cite{FM} and L\"oh-Sauer \cite{LS}). Theorem \ref{nerve_theorem}(1) is originally due to Ivanov (see \cite{Iv}); alternative proofs of this result can also be found in \cite{FM, LS}. More recently, Ivanov \cite{Iv2} obtained generalizations of Theorem \ref{nerve_theorem}(1) to more general classes of covers. 

\smallskip

In this paper we present another approach to Theorem \ref{nerve_theorem} which also leads to stronger versions of the theorem and generalizations. Let us outline the general idea of this approach in the specific context of Theorem \ref{nerve_theorem}. Note that for any open cover $\U = \{U_i \}_{i \in I}$ of $X$, there is an associated diagram of topological spaces
$$F_{\U} \colon \I_{\U} \to \Top, \ \sigma \mapsto U_{\sigma}: = \bigcap_{i \in \sigma} U_i.$$
A fundamental property of $F_{\U}$ is that it is a homotopy colimit diagram, i.e., the canonical map $$\hocolim_{\I_{\U}} F_{\U} \rightarrow X$$
is a weak homotopy equivalence \cite{DI}.  In addition, for the constant diagram at the one-point space $\ast$,
$$\delta(\ast) \colon \I_{\U} \to \Top, \ \sigma \mapsto \ast,$$
there is a canonical weak homotopy equivalence 
$\hocolim_{\I_{\U}} \delta(\ast) \simeq |\I_{\U}|.$
The natural transformation $F_{\U} \to \delta(\ast)$ yields the canonical map (that appears in Theorem \ref{nerve_theorem})
$$p_{F_{\U}} \colon X \simeq \hocolim_{\I_{\U}} F_{\U} \to \hocolim_{\I_{\U}} \delta(\ast) \simeq |\I_{\U}|.$$
Based on the naturality of $c_X$ and using the fact that singular cohomology is excisive, it is possible to obtain factorizations of $c_X$ from natural factorizations of the maps $U_{\sigma} \subseteq X$. In the case of Theorem \ref{nerve_theorem}, the key assumption that $\U$ is amenable implies that for each inclusion $U_{\sigma} \subseteq X$, $\sigma \in \I_{\U}$, the natural Moore--Postnikov factorization (at $\pi_1$)
$$U_{\sigma} \to V_{\sigma} \to X$$
has the property that the bounded cohomology of $V_{\sigma}$ is concentrated in degree $0$ -- this is a consequence of Gromov's fundamental \emph{Mapping Theorem} \cite{Gr, Iv, FM}. Moreover, if $U_{\sigma}$ is path-connected, then $V_{\sigma}$ has the (bounded) cohomology of the point.  These properties yield information about the homotopy limit of the diagram of bounded cochain complexes $\sigma \mapsto C^{\bullet}_b(V_{\sigma}; \R)$. In particular, passing to the homotopy limit (for $\sigma \in \I_{\U}$) in the diagrams of cochain complexes
$$
\xymatrix{
C^{\bullet}_b(X; \R) \ar[d]  \ar@/^3pc/[rrdd]^{c_X} & & \\
C^{\bullet}_b(V_{\sigma}; \R) \ar[r]^{c_{V_{\sigma}}} & C^{\bullet}(V_{\sigma};\R) \ar[d] & \\
& C^{\bullet}(U_{\sigma};\R) &   C^{\bullet}(X;\R) \ar[l]
}
$$
we obtain a canonical factorization of $c_X$ with the required properties. This approach to Theorem \ref{nerve_theorem}, based on homotopy colimits and the  properties of $c_X$, leads to refinements of Theorem \ref{nerve_theorem} in several directions:
\begin{itemize}
\item[(a)] The context of Theorem \ref{nerve_theorem} extends to general homotopy (or $\R$-cohomology) colimit decompositions of $X$. For each such homotopy colimit diagram that is equipped with a factorization (in the sense of Section \ref{sec_factorizations}), there is a corresponding canonical factorization of $c_X$ (Proposition \ref{factorization_general}). 

Moreover, for factorizations of the homotopy colimit diagram which satisfy certain vanishing properties, our generalized versions of the \emph{Covering theorem} (Theorem \ref{nerve_theorem_general1}) and of the \emph{Vanishing theorem} (Theorem \ref{nerve_theorem_general2}) show specific properties of the induced factorization of $c_X$ analogous to Theorem \ref{nerve_theorem}. The precise connection with Theorem \ref{nerve_theorem} is discussed in Section \ref{amenable_covers}.    
\item[(b)] The method yields canonical factorizations of the comparison map $c_X$ in the homotopy theory of cochain complexes before passing to cohomology. 

Working directly with cochain complexes provides the statements with more structure and allows for new variations of the results by direct application of homotopy-theoretic constructions. In Section \ref{pullback_fact}, we discuss some examples of this in the context of \emph{parametrized factorizations} (Subsections \ref{general_parametrized} and \ref{pullback_covers})  as well as for relative bounded cohomology (Subsection \ref{subsec:rel-cov-van}). 
\item[(c)] The same method applies similarly to factorizations of homotopy (or $\R$-cohomology) colimit decompositions which satisfy bounded acyclicity only in appropriate ranges of degrees. Specifically, our version of the \emph{Covering theorem} (Theorem \ref{nerve_theorem_general1}) treats also \emph{boundedly $k$-acyclic factorizations} and our version of the \emph{Vanishing theorem} (Theorem \ref{nerve_theorem_general2}) applies also to \emph{boundedly $k$-truncating factorizations}. Amenable covers give rise to $0$-truncating factorizations (see Section \ref{amenable_covers}). In view of the renewed interest and the recent advances in bounded acyclicity, these technical refinements of Theorem \ref{nerve_theorem} could potentially prove useful for future applications. 
\end{itemize}
As the general method is essentially based only on the general properties of the comparison map $c_X$, similar arguments can also be used for other such examples of natural transformations. The main ingredient of our approach is to view the comparison map as a natural map of cochain complexes and use its homotopy-theoretic properties in order to identify it as the \emph{coassembly map for bounded cohomology} (see Subsection \ref{sec:comparison_map}).  For another related example where the same method can be applied, let us mention the dual context of the comparison map for $\ell^1$-homology
$$H_*(X; \R) \to H^{\ell^1}_*(X;\R)$$
which can be treated similarly and the proofs are essentially the same. Further generalizations are discussed briefly in Subsection \ref{generalizations}.

Since several proofs of Theorem \ref{nerve_theorem} (and its variations) are available in the literature, this paper makes no great claim to originality. The main purpose of the paper is to show how a specific homotopy-theoretic viewpoint offers another perspective on Theorem \ref{nerve_theorem} and leads naturally to stronger versions of this result with short proofs; moreover, we hope that the methods might be of further interest in the study of the homotopy-theoretic aspects of bounded cohomology and its applications. 

\medskip

\noindent \textbf{Acknowledgements.} I would like to thank Clara L\"oh and Marco Moraschini for many interesting discussions about bounded cohomology and for their useful comments on this work. This work was partially supported by \emph{SFB 1085 -- Higher Invariants} (Universit\"at Regensburg) funded by the DFG. 

\section{Preliminaries} \label{sec:prelim}

We briefly review the context of the comparison map for bounded cohomology using the language $\infty$-categories; the theory of $\infty$-categories is 
not necessary for obtaining our results, but it will allow us here, as in many other of its applications, to identify and focus on the essential homotopy-theoretic features of the problem. We refer the interested reader to \cite[Ch. 1]{HTT} for an introduction to $\infty$-categories. 
 
\subsection{Spaces and chain complexes} Let $\Sp$ denote the $\infty$-category of spaces (see, for example, \cite{HTT}). It will be convenient to use 
the following model for $\Sp$. Let $\Top$ denote the (ordinary) category of topological spaces and let $\W_{\mathrm{whe}}$ be the class of weak homotopy equivalences in $\Top$. We define 
$$\Sp : = \Top[\W_{\mathrm{whe}}^{-1}]$$
to be the $\infty$-categorical localization of $\Top$ at $\W_{\mathrm{whe}}$ \cite[Ch. 7]{Ci}, \cite{Hi}.  There is no substantive difference between the approach using $(\Top, \W_{\mathrm{whe}})$ (as a category with weak equivalences or a model category) and the associated $\infty$-category $\Sp$.

Let $\Ch_{\R}$ denote the (ordinary) category of chain complexes of $\R$-vector spaces and let $\W_{\mathrm{qi}}$ be the class of quasi-isomorphisms in $\Ch_ {\R}$. We consider similarly the model for the derived $\infty$-category $\ch(\R)$ that is given by the $\infty$-categorical localization of $\Ch_{\R}$ at $\W_{\mathrm{qi}}$:
$$\ch(\R) : = \Ch_{\R}[\W_{\mathrm{qi}}^{-1}].$$

\subsection{(Co)limit diagrams} Given a simplicial set $\I$ (e.g., the nerve of a small category), we denote by $\I^{\triangleright} := \I \star \Delta^0$ the associated simplicial set (or the nerve of the associated category) that is defined by joining a terminal object $\pt$ to $\I$. 

A cone on a diagram $F \colon \I \to \Sp$ is a diagram $\widetilde{F} \colon \I^{\triangleright} \to \Sp$ extending $F$. A colimit diagram $\widetilde{F} \colon \I^{\triangleright} \to \Sp$ is a cone on $F = \widetilde{F}_{|\I} \colon \I \to \Sp$ which is initial in the $\infty$-category $\Sp_{F/}$ of cones on the diagram $F$. Limits in $\Sp$ (or in any other $\infty$-category) are defined dually. For a colimit diagram $\widetilde{F} \colon \I^{\triangleright} \to \Sp$ of $F$, we will sometimes refer to the cone object $\widetilde{F}(\pt) = X \in \Sp$ as the colimit of $F$ and denote this by $\colim_{\I} F$. 
Every small diagram in $\Sp$ (or in $\ch(\R)$) admits a (co)limit (see \cite{Ci, HTT}). 

We say that a cone $\widetilde{F} \colon \I^{\triangleright} \to \Sp$ is an \textit{$\HR$-colimit diagram} on $F = \widetilde{F}_{|\I} \colon \I \to \Sp$ if the canonical map in $\Sp$
$$\colim_{\I} F \to \widetilde{F}(\pt)$$
induces an isomorphism in singular $\R$-cohomology. A colimit diagram is clearly an $\HR$-colimit diagram.

Given a diagram $F \colon \I \to \Sp$, a (colimit, $\HR$-colimit) diagram $\widetilde{F} \colon \I^{\triangleright} \to \Sp$ which extends $F$ and has cone object $X \in \Sp$ determines (and is determined by) its adjoint diagram in the $\infty$-category $\Sp_{/X}$ of spaces over $X$,  also denoted by 
$$\widetilde{F} \colon \I \to \Sp_{/X}, \ i \mapsto (F(i) \to X).$$
We will call a diagram in $\Sp_{/X}$ that arises in this way from a colimit diagram in $\Sp$ (resp. $\HR$-colimit diagram) again a \emph{colimit diagram}  (resp. $\HR$\emph{-colimit diagram}). 

We refer to \cite{Ci, HTT} for a detailed exposition of the properties of (co)limits in higher category theory. There is no essential difference between such colimit diagrams and homotopy colimit diagrams in $\Top$ in the standard sense (see \cite[Ch. 4]{HTT} and \cite[Ch. 7]{Ci} for an analysis of this comparison). We also recommend \cite{Du} and \cite{DHKS} for the theory of homotopy (co)limits in homotopical algebra.

\begin{example} For a simplicial set $\I$, we denote the associated space ($\infty$-groupoid) by $|\I|$;  this has the weak homotopy type of (the classifying space of) $\I$ (when $\I$ corresponds to the nerve of an ordinary category).  The colimit of the constant diagram at the one-point/contractible space, 
$$\delta(\ast) \colon \I \to \Sp, \ \ i \mapsto \ast,$$ is canonically identified with $|\I|$, i.e., $\colim_{\I} \delta(\ast) \simeq |\I|.$

\noindent For every diagram $F \colon \I \to \Sp$, there is an obvious natural transformation $F \rightarrow \delta(\ast)$ which induces the canonical map: 
$$p_F \colon \colim_{\I} F \to \colim_{\I} \delta(\ast) \simeq |\I|.$$
\end{example}

\begin{example}[Open covers] \label{covering-colim} Let $\U = \{U_i \}_{i \in I}$ be an open cover of a topological space $X$. For a subset $\sigma \subseteq I$, we write $U_{\sigma} := \bigcap_{i \in \sigma} U_i \subseteq X$.  Let $\I_{\U}$ denote the (nerve of the) poset of finite subsets $\sigma \subseteq I$ such that $U_{\sigma} \neq \varnothing$, ordered by reverse inclusion of subsets.  Passing to the $\infty$-category of spaces, there is an associated diagram:
$$\widetilde{F}_{\U} \colon \I_{\U} \to \Sp_{/X}, \ \sigma \mapsto (U_{\sigma} \xrightarrow{i_{\sigma}} X).$$
It is a fundamental property of the diagram $\widetilde{F}_{\U}$ that it defines a colimit diagram, that is, the canonical map 
\begin{equation*} \label{di} 
\colim_{\sigma \in \I_{\U}} U_{\sigma} \rightarrow X
\end{equation*}
is an equivalence in $\Sp$ (see \cite{DI}). 
\end{example}

\subsection{Singular cohomology} Let $C^{\bullet}(-;\R) \colon \Top^{\op} \to \Ch_{\R}$ denote the usual singular $\R$-cochain complex functor. It will be convenient here to regard the singular cochain complex as a chain complex concentrated in nonpositive degrees.  Since $C^{\bullet}(-;\R)$ sends weak homotopy equivalences to quasi-isomorphisms, we obtain canonically an induced functor
$$C^{\bullet}(-; \R) \colon \Sp^{\op} \to \ch(\R).$$ 
It follows from classical results that this functor is \emph{excisive}, i.e., $C^{\bullet}(-;\R)$ sends colimits in $\Sp$ to limits in $\ch(\R)$ -- the case 
of coproducts is obvious and the case of pushouts follows from the Mayer-Vietoris sequence for singular cohomology. 

\subsection{Bounded cohomology} Let $C^{\bullet}_b(-; \R) \colon \Top^{\op} \to \Ch_{\R}$ denote the functor of the bounded cochain complex \cite{Gr, Iv}, which we regard here as a chain complex concentrated in nonpositive degrees. We recall that a singular cochain 
$$\phi \colon C_n(X; \R) \to \R$$ 
is \emph{bounded} if the subset $\{\phi(\sigma) \in \R \ | \ \sigma \colon \Delta^n \to X\} \subseteq \R$ is bounded. For $n \geq 0$, the bounded cohomology of $X$ is defined as follows: 
$$H^{n}_b(X; \R) := H_{-n}(C^{\bullet}_b(X; \R)).$$ 
Since $C^{\bullet}_b(-; \R)$ sends weak homotopy equivalences to quasi-isomorphisms \cite{Iv}, we obtain canonically an induced functor
$$C^{\bullet}_b(-; \R) \colon \Sp^{\op} \to \ch(\R).$$
(Note that a weak homotopy equivalence of topological spaces induces a simplicial homotopy equivalence between the corresponding singular sets, since these are Kan complexes. In turn, this induces a chain homotopy equivalence between the singular cochain complexes which also restricts to bounded cochains.)

\subsection{The comparison map} \label{sec:comparison_map} The inclusion of bounded cochains $C^{\bullet}_b(-;\R)$ into all singular cochains $C^{\bullet}(-;\R)$ induces a natural transformation between the functors $C^{\bullet}_b(-; \R), C^{\bullet}(-;\R) \colon \Sp^{\op} \to \ch(\R)$,
$$c \colon C^{\bullet}_b(-;\R) \to C^{\bullet}(-;\R).$$
We will refer to this as the \emph{comparison map} (for bounded cohomology). We also denote the induced comparison maps on cohomology by 
$$c^n \colon H^n_b(-; \R) \to H^n(-;\R).$$

A natural transformation $C^{\bullet}_b(-; \R) \to H$ to an excisive functor $H \colon \Sp^{\op} \to \ch(\R)$ is determined essentially uniquely by its restriction to $X = \ast$; this is because every (topological) space can be written canonically as a (homotopy) colimit of (weakly) contractible spaces. The comparison map for bounded cohomology is an equivalence for $X = \ast$, that is, the map 
$$c_{\{\ast\}} \colon C^{\bullet}_b(\ast ; \R) \to C^{\bullet}(\ast ;\R)$$ 
is an equivalence in $\ch(\R)$ (it is, in fact, an isomorphism of chain complexes in $\Ch_{\R}$). This map $c_{\{\ast\}}$ together with the fact that $C^{\bullet}(-;\R)$ is excisive completely determine the comparison map $c$. Note that the comparison map is not an equivalence in general exactly because $C^{\bullet}_b(-;\R)$ fails to be excisive. But $c$ is the initial transformation to an excisive functor. More generally,
for any cocomplete (resp. complete) $\infty$-category $\D$, there is an adjunction of (not necessarily locally small) $\infty$-categories:
$$(-)_{\%} \colon \D \rightleftarrows \mathrm{Fun}(\Sp, \D)\colon \mathrm{ev} \ \ (\text{resp. } \mathrm{ev} \colon \mathrm{Fun}(\Sp^{\op}, \D) \rightleftarrows \D \colon (-)^{\&})$$
where the right (resp. left) adjoint $\mathrm{ev}$ is given by the evaluation at the contractible space $\ast$ and the associated counit (resp. unit) transformation 
$$\smallint \colon (\mathrm{ev}(F))_{\%} \to F \ \ (\text{resp. } \nabla \colon F \to (\mathrm{ev}(F))^{\&})$$
is the \emph{assembly map for $F$}. The functor $(-)^{\&}$ (resp. $(-)_{\%}$) is fully faithful and its essential image consists of the excisive functors. The natural transformation $c$ is an instance of an assembly map in the dual context of contravariant functors 
-- $c$ is the \emph{coassembly map for bounded cohomology}. See \cite{WW} for more details about the general construction of assembly maps and their properties, and \cite[6.3-6.4]{Ci} for related general results in the language of $\infty$-categories. 
  
\section{Factorizations} \label{sec_factorizations}

\subsection{General factorizations} Let $X \in \Sp$ be a space and let $\widetilde{F} \colon \I \to \Sp_{/X}$ be a diagram of spaces over $X$. A \emph{factorization of $\widetilde{F} \colon \I \to \Sp_{/X}$} consists of a diagram $\widetilde{G} \colon \I \to \Sp_{/X}, \ i \mapsto (G(i) \to X),$ together with a natural transformation $\eta \colon \widetilde{F} \to\widetilde{G}$. 

\smallskip

The following proposition shows that factorizations of $\HR$-colimit diagrams yield factorizations of the comparison map for bounded cohomology.

\begin{proposition} \label{factorization_general}
Let $X \in \Sp$ be a space and let $\widetilde{F} \colon \I \to \Sp_{/X}, \ i \mapsto (F(i) \to X),$
be an $\HR$-colimit diagram. For every factorization $\eta \colon \widetilde{F} \to \widetilde{G}$, where 
$$\G \colon \I \to \Sp_{/X}, \ i \mapsto (G(i) \to X),$$
is a diagram of spaces over $X$, there is a canonical factorization of $c_X$ in $\ch(\R)$:
$$
\xymatrix{
C^{\bullet}_b(X; \R) \ar[r]^{c_X} \ar[d] & C^{\bullet}(X; \R). \\
\lim_{\I^{\op}}C^{\bullet}_b(G(-); \R) \ar[r] & \lim_{\I^{\op}} C^{\bullet}(G(-); \R) \ar[u]  
}
$$
(The maps will be specified in the proof below.)
\end{proposition}
\begin{proof}
There is a canonical diagram in $\ch(\R)$:
$$\small{
\xymatrix{
C^{\bullet}_b(X; \R) \ar[d] \ar[rr]^{c_X} && C^{\bullet}(X; \R) \ar[d] \\
C^{\bullet}_b(\colim_{\I} G; \R) \ar[d] \ar[rr] && C^{\bullet}(\colim_{\I} G; \R) \ar[d] \\
C^{\bullet}_b(\colim_{\I} F; \R) \ar[rr] && C^{\bullet}(\colim_{\I} F; \R). \\
}}
$$
The horizonal maps are induced by the natural comparison map; the vertical maps are given by the functoriality of $C^{\bullet}_b(-;\R)$ and $C^{\bullet}(-;\R)$. The right vertical composite
$$C^{\bullet}(X; \R) \xrightarrow{\simeq} C^{\bullet}(\colim_{\I} F; \R)$$ is an equivalence because the canonical map $\colim_{\I} F \rightarrow X$ is an $\R$-cohomology equivalence by assumption. (Note that the left vertical composite will also be an equivalence and the bounded cochain complex $C^{\bullet}_b(X; \R)$ will be a retract of $C^{\bullet}_b(\colim_{\I} G; \R)$ if $\widetilde{F}$ is actually a colimit diagram.) Using the fact that $C^{\bullet}(-;\R)$ is excisive 
and the naturality of the comparison map $c$, the bottom square is identified with the composite diagram: 
$$\small{
\xymatrix{
C^{\bullet}_b(\colim_{\I} G; \R) \ar[d]  \ar[r] & \lim_{\I^{\op}} C^{\bullet}_b(G(-); \R) \ar[d] \ar[r]^(0.36){c_{G(-)}} & \lim_{\I^{\op}} C^{\bullet}(G(-); \R) \xleftarrow{\simeq} C^{\bullet}(\colim_{\I} G; \R)  \ar@<8ex>[d] \ar@<-8ex>[d] \\
C^{\bullet}_b(\colim_{\I} F; \R) \ar[r] & \lim_{\I^{\op}} C^{\bullet}_b(F(-); \R) \ar[r]^(0.36){c_{F(-)}} & \lim_{\I^{\op}} C^{\bullet}(F(-); \R) \xleftarrow{\simeq} C^{\bullet}(\colim_{\I} F; \R). \\
}}
$$
The left and right horizontal maps are the canonical maps to the respective limits; the middle maps are induced by the natural comparison map. 
Then the combination of the two diagrams produces the required factorization. 
\end{proof}

\begin{example}[Tautological factorization] \label{ex3}
Let $\widetilde{F} \colon \I \to \Sp_{/X}, \ i \mapsto (F(i) \to X),$ be an $\HR$-colimit diagram and let $\eta \colon \widetilde{F} = \widetilde{F}$ be the identity. Then we obtain a canonical factorization of the comparison map $c_X$ in $\ch(\R)$:
$$
\xymatrix{
C^{\bullet}_b(X; \R) \ar[r]^{c_X} \ar[d] & C^{\bullet}(X; \R) \\
\lim_{\I^{\op}}C^{\bullet}_b(F(-); \R) \ar[ur]_q & 
}
$$
where the vertical map is the canonical map to the limit and $q$ is the composition
$$\lim_{\I^{\op}} C^{\bullet}_b(F(-); \R) \xrightarrow{c_{F(-)}} \lim_{\I^{\op}} C^{\bullet}(F(-); \R) \simeq C^{\bullet}(\colim_{\I} F; \R) \simeq C^{\bullet}(X; \R).$$
More generally, note that the factorization of $c_X$ in Proposition \ref{factorization_general} happens essentially on the side of bounded cohomology, i.e., the comparison map factors canonically through the limit of the \emph{bounded} cohomologies of a diagram 
of spaces (which depends on $\widetilde{G}$). 
\end{example}

\begin{example}[Moore--Postnikov truncation] \label{Moore-Postnikov}
We recall that every map $f \colon Y \to X$ in $\Sp$ admits canonical \emph{Moore--Postnikov $n$-truncations} \cite[Ch. IX]{Wh}, \cite[IV.3]{GJ}, \cite[6.5]{HTT}; in the case of $\pi_1$, this is a natural factorization of $f$, 
$$Y \xrightarrow{p_f} \mathcal{T}(f) \xrightarrow{j_f} X,$$
which is characterized essentially uniquely by the following properties:
\begin{itemize}
\item[(i)] $p_f$ induces a bijection $\pi_0(Y) \xrightarrow{\cong} \pi_0(\mathcal{T}(f))$. Moreover, for every point $y$ in  $Y$, the induced homomorphism
$$\pi_1(p_f, y) \colon \pi_1(Y, y) \to \pi_1(\mathcal{T}(f), p_f(y))$$
is surjective.  
\item[(ii)] For every point $z$ in $\mathcal{T}(f)$, the homomorphism induced by $j_f$,
$$\pi_1(j_f, z) \colon \pi_1(\mathcal{T}(f), z) \to \pi_1(X, j_f(z)),$$ 
is injective. Moreover, for every $n > 1$, $j_f$ induces an isomorphism 
$$\pi_n(j_f, z) \colon \pi_n(\mathcal{T}(f), z) \xrightarrow{\cong} \pi_n(X, j_f(z)).$$
\end{itemize}
Let $\widetilde{F} \colon \I \to \Sp_{/X}, \ i \mapsto (f(i) \colon F(i) \to X),$ be an $\HR$-colimit diagram. Applying the Moore--Postnikov truncation (at $\pi_1$) to the maps 
$(f(i) \colon F(i) \rightarrow X)$ yields a new diagram 
$$\mathcal{T}(\widetilde{F}) \colon \I \to \Sp_{/X}, \ i \mapsto (\mathcal{T}(f(i)) \xrightarrow{j_{f(i)}} X),$$
together with a factorization $p \colon \widetilde{F} \to \mathcal{T}(\widetilde{F}),$ whose components are given by the maps $p_{f(i)}$ in $\Sp_{/X}$. 
Therefore, by Proposition \ref{factorization_general}, we have a canonical factorization in $\ch(\R)$ as follows:
$$
\xymatrix{
C^{\bullet}_b(X; \R) \ar[r]^{c_X} \ar[d] & C^{\bullet}(X; \R). \\
\lim_{\I^{\op}} C^{\bullet}_b(\mathcal{T}(f(-)); \R) \ar[r] & \lim_{\I^{\op}} C^{\bullet}(\mathcal{T}(f(-)); \R) \ar[u]  
}
$$
\end{example}

\subsection{(Boundedly) acyclic/truncating factorizations} 
The factorization of the comparison map $c_X$ in Proposition \ref{factorization_general} will only be useful in the case of factorizations $\eta \colon \widetilde{F} \to \widetilde{G}$, for which we can determine some interesting properties of the limit of bounded cochain complexes $i \mapsto C^{\bullet}_b(G(i); \R)$. 

We will consider two special classes of factorizations: \emph{boundedly ($k$-)acyclic} and \emph{boundedly ($k$-)truncating} factorizations. 

\begin{definition} \label{def-acyclic-fact}
Let $X \in \Sp$ be a space, let $\widetilde{F} \colon \I \to \Sp_{/X}$ be a diagram, and let $k \geq 0$ be an integer. 
\begin{itemize}
\item[(1)] We say that $\eta \colon \widetilde{F} \to \widetilde{G}$ is an \emph{acyclic factorization of $\widetilde{F}$} (\emph{with respect to $C^{\bullet}_b(-;\R)$}) if the diagram 
$$\widetilde{G} \colon \I \to \Sp_{/X}, \ i \mapsto (G(i) \to X),$$ satisfies the following property: for each $i \in \I$,
the canonical map
$$H^n_b(\ast; \R) \xrightarrow{H^n_b(G(i) \to \ast; \R)} H^n_b(G(i); \R)$$
is an isomorphism for every $n \geq 0$.  
\item[(2)] A factorization $\eta \colon \widetilde{F} \to \widetilde{G}$ is a \emph{k-acyclic factorization of $\widetilde{F}$} (\emph{with respect to $C^{\bullet}_b(-;\R)$}) if the diagram 
$$\widetilde{G} \colon \I \to \Sp_{/X}, \ i \mapsto (G(i) \to X),$$ satisfies the following property: for each $i \in \I$,
the canonical map
$$H^n_b(\ast; \R) \xrightarrow{H^n_b(G(i) \to \ast; \R)} H^n_b(G(i); \R)$$
is an isomorphism for every $n \leq k$ (and a monomorphism in degree $k+1$).  
\end{itemize}
\end{definition}

The factorization in Proposition \ref{factorization_general} has interesting implications when we restrict to ($k$-)acyclic factorizations. As we will see in Section \ref{amenable_covers}, the following theorem is a refinement and generalization of the \emph{Covering theorem} (Theorem \ref{nerve_theorem}(1)). 

\begin{theorem} \label{nerve_theorem_general1}
Let $X \in \Sp$ be a space and let $\widetilde{F} \colon \I \to \Sp_{/X}, \ i \mapsto (F(i) \to X),$ be an $\HR$-colimit diagram.
\begin{itemize}
\item[(1)] If $\eta \colon \widetilde{F} \to \widetilde{G}$ is an acyclic factorization of $\widetilde{F}$, then there is a canonical factorization of $c_X$ in $\ch(\R)$:
$$
\xymatrix{
C^{\bullet}_b(X; \R) \ar[rr]^{c_X} \ar[dr] && C^{\bullet}(X; \R) \\
& C^{\bullet}(|\I|; \R) \ar[ur]_{\rho_F} &
}
$$
where $\rho_F \colon C^{\bullet}(|\I|; \R) \xrightarrow{C^{\bullet}(p_{F}; \R)} C^{\bullet}(\colim_{\I} F; \R) \simeq C^{\bullet}(X; \R).$ 
\item[(2)] If $\eta \colon  \widetilde{F} \to \widetilde{G}$ is a $k$-acyclic factorization of $\widetilde{F}$, for some $k \geq 0$, then for every $n \leq k$ there is a canonical factorization of $c^n_X$:
$$
\xymatrix{
H^n_b(X; \R) \ar[rr]^{c^n_X} \ar[dr] && H^n(X; \R) \\
& H^n(|\I|; \R) \ar[ur]_{\rho^n_F} &
}
$$
where $\rho^n_F$ is induced by $\rho_F$ as defined in (1).
\end{itemize}
\end{theorem}
\begin{proof}
(1): Using the acyclicity of $\eta$ and the equivalence $c_{\{\ast\}} \colon C^{\bullet}_b(\ast; \R) \simeq C^{\bullet}(\ast;\R)$, we obtain the following canonical identifications in $\ch(\R)$
$$\lim_{\I^{\op}} C^{\bullet}_b(G(-);\R) \simeq \lim_{\I^{\op}} C^{\bullet}_b(\ast; \R) \simeq \lim_{\I^{\op}} C^{\bullet}(\ast; \R).$$
Moreover, since $C^{\bullet}(-; \R)$ is excisive, we also have canonical identifications:
$$\lim_{\I^{\op}} C^{\bullet}(\ast; \R) \simeq C^{\bullet}(\colim_{\I} \delta(\ast); \R) \simeq  C^{\bullet}(|\I|; \R).$$
Then the desired factorization follows from Proposition \ref{factorization_general}. Moreover, the identification of $\rho_F$ can be seen from the following diagram in $\ch(\R)$:
$$
\xymatrix{
\lim_{\I^{\op}} C^{\bullet}_b(G(i); \R) \ar[r]  & \lim_{\I^{\op}} C^{\bullet}(G(i); \R) \simeq C^{\bullet}(\colim_{\I} G; \R) \rightarrow  C^{\bullet}(\colim_{\I} F; \R)  \\
\lim_{\I^{\op}} C^{\bullet}_b(\ast; \R) \ar[r]^(0.3){\simeq} \ar[u]^{\simeq} & \lim_{\I^{\op}} C^{\bullet}(\ast ;\R) \simeq C^{\bullet}(|\I|; \R) \xrightarrow{\rho_F}   C^{\bullet}(X;\R). \ar@<3ex>[u]  \ar@<17ex>[u]  \ar@<-17ex>[u]_{\simeq}
}
$$
(Note that the fact that the bottom left map is an equivalence is actually not necessary for the proof.)

\smallskip

\noindent (2): Similarly to the proof of (1), we consider the canonical map in $\ch(\R)$
$$\lim_{\I^{\op}} C^{\bullet}_b(\ast; \R) \longrightarrow \lim_{\I^{\op}} C^{\bullet}_b(G(-); \R)$$
which is induced by the natural transformation $G \to \delta(\ast)$. Using that the factorization $\eta$ is $k$-acyclic, it follows that this map also induces an isomorphism in cohomology in degrees $\leq k$ (and a monomorphism in degree $k+1$), since this  property is closed under taking limits in $\ch(\R)$. Then the desired factorization follows from Proposition \ref{factorization_general} and the same diagram in $\ch(\R)$ as above (after passing to cohomology). 
\end{proof}

\begin{example}[Boundedly acyclic $\HR$-colimit diagrams] \label{bounded-acyclic-colimits}
Let $\widetilde{F} \colon \I \to \Sp_{/X}, \ i \mapsto (F(i) \to X),$ be an $\HR$-colimit diagram and let $\eta \colon \widetilde{F} = \widetilde{F}$ be the tautological factorization (Example \ref{ex3}). Suppose that the values of $F$ are \emph{boundedly acyclic}, that is, for each $i \in \I$, the map
$$H^n_b(\ast; \R) \xrightarrow{H^n_b(F(i) \to \ast; \R)} H^n_b(F(i); \R)$$
is an isomorphism for $n \geq 0$. In other words, $\eta$ is an acyclic factorization of $\widetilde{F}$, so we obtain by Proposition \ref{factorization_general} and Theorem \ref{nerve_theorem_general1} a canonical factorization of the comparison map $c_X$ in $\ch(\R)$
$$
\xymatrix{
C^{\bullet}_b(X; \R) \ar[r]^{c_X} \ar[d] & C^{\bullet}(X; \R). \\
\lim_{\I^{\op}} C^{\bullet}_b(F(-); \R) \ar[r]^(.6){\simeq} & C^{\bullet}(|\I|; \R) \ar[u]_{\rho_F} 
}
$$
\end{example}

\begin{remark} \label{acyclic-fact-gromov} Let $\widetilde{F} \colon \I \to \Sp_{/X}, \ i \mapsto (F(i) \to X),$ be an $\HR$-colimit diagram and let $\eta \colon \widetilde{F} \to \widetilde{G}$ be a factorization, where $\widetilde{G} \colon \I \to \Sp_{/X}, i \mapsto (G(i) \to X)$. As a consequence of Gromov's \emph{Mapping Theorem} \cite{Gr, Iv}, the factorization $\eta$ is acyclic if and only if the space $G(i)$ is connected and its fundamental group has trivial bounded 
cohomology for all $i \in \I$ (see also \cite{MR}). This shows that the factorizations of Example \ref{Moore-Postnikov} have a special importance for the applications of Theorem \ref{nerve_theorem_general1}. We will explore this further in Section \ref{amenable_covers}.
\end{remark}

On the other hand, the notion of a boundedly $k$-truncating factorization requires vanishing properties in the complementary range of degrees. 

\begin{definition} \label{def-truncating-fact}
Let $X \in \Sp$ be a space, let $\widetilde{F} \colon \I \to \Sp_{/X}$ be a diagram, and let $k \geq 0$ be an integer. 
A factorization $\eta \colon \widetilde{F} \to \widetilde{G}$ is a \emph{k-truncating factorization of $\widetilde{F}$} (\emph{with respect to $C^{\bullet}_b(-;\R)$}) if the diagram 
$$\widetilde{G} \colon \I \to \Sp_{/X}, \ i \mapsto (G(i) \to X),$$ satisfies the following property: $H^n_b(G(i); \R) \cong 0$ for every $n > k$ and each $i \in \I$.
\end{definition}

The application of the factorization in Proposition \ref{factorization_general} to boundedly truncating factorizations yields the following vanishing result for the comparison map. As we will see in Section \ref{amenable_covers}, this result is a refinement and generalization of the \emph{Vanishing theorem} (Theorem \ref{nerve_theorem}(2)). 

\begin{theorem} \label{nerve_theorem_general2}
Let $X \in \Sp$ be a space and let $\widetilde{F} \colon \I \to \Sp_{/X}, \ i \mapsto (F(i) \to X),$ be an $\HR$-colimit diagram. 
\begin{itemize}
\item[(1)] If $\eta \colon \widetilde{F} \to \widetilde{G}$ is a $0$-truncating factorization of $\widetilde{F}$ and $\I$ is (Joyal) equivalent to a simplicial set of dimension $< m$, then the comparison map $c_X$ vanishes in cohomology in degrees $\geq m$.
\item[(2)] More generally, if $\eta \colon \widetilde{F} \to \widetilde{G}$ is a $k$-truncating factorization of $\widetilde{F}$ and $\I$ is (Joyal) equivalent to a simplicial set of dimension $d$, then the comparison map $c_X$ vanishes in cohomology in degrees $ > d + k$.
\end{itemize}
\end{theorem}
\begin{proof} (1): Since $\eta$ is a $0$-truncating factorization, the map $G(i) \to \pi_0(G(i))$, $i \in \I$, induces an equivalence in $\ch(\R)$
$$C^{\bullet}_b(G(i); \R) \simeq C^{\bullet}_b(\pi_0(G(i)); \R).$$
Then, combining the factorization in Proposition \ref{factorization_general} with the diagram
$$
\xymatrix{ 
\lim_{\I^{\op}}C^{\bullet}_b(G(-); \R) \ar[r] & \lim_{\I^{\op}} C^{\bullet}(G(-); \R), \\ 
\lim_{\I^{\op}}C^{\bullet}_b(\pi_0(G(-)); \R) \ar[u]^{\simeq} \ar[r] & \lim_{\I^{\op}} C^{\bullet}(\pi_0(G(-)); \R) \ar[u] 
}
$$
we obtain a factorization of the comparison map $c_X$ through the limit of singular (bounded) cochain complexes 
$$D^{\bullet}:= \lim_{\I^{\op}} C^{\bullet}(\pi_0(G(-)); \R) \ \ (\text{resp. } \lim_{\I^{\op}} C^{\bullet}_b(\pi_0(G(-)); \R)).$$ 
Both of these are (homotopy) limits of ($\Z$-graded) chain complexes which are concentrated in degree $0$ (up to quasi-isomorphism). Suppose without loss of generality that $\I$ has dimension $< m$. Then the homology of the limit $D^{\bullet}$ is concentrated in nonpositive degrees $k$ with $|k| < m$; this follows by induction on the dimension of $\I$ using \cite[Proposition 4.4.2.2]{HTT} and standard facts from homological algebra. In other words, the cohomology of $D^{\bullet}$ vanishes in cohomological degrees $\geq m$, as required. 

\smallskip

\noindent (2): The proof is similar to (1).  By Proposition \ref{factorization_general}, the comparison map $c_X$ factors canonically through the limit 
$$\lim_{\I^{\op}} C^{\bullet}_b(G(-); \R).$$
By assumption, this is a limit of bounded cochain complexes, each of whose cohomology is concentrated in degrees $\leq k$ -- here we are using the cohomological grading convention. We may suppose without loss of generality that $\I$ has dimension $d$. Then, by induction on the dimension of $\I$, using \cite[Proposition 4.4.2.2]{HTT} and standard facts from homological algebra, it follows that the cohomology of the limit is concentrated in degrees $\leq d + k$. (Note that the proof shows the stronger statement that $c_X$ factors through a cochain complex concentrated in degrees $\leq d + k$.)
\end{proof}

\begin{remark} \label{acyclic-fact-gromov2} Let $\widetilde{F} \colon \I \to \Sp_{/X}, \ i \mapsto (F(i) \to X),$ be an $\HR$-colimit diagram and let $\eta \colon \widetilde{F} \to \widetilde{G}$ be a factorization, where $\widetilde{G} \colon \I \to \Sp_{/X}, i \mapsto (G(i) \to X)$. Similarly to Remark \ref{acyclic-fact-gromov}, if the factorization $\eta$ is $0$-truncating, then each component of $G(i)$ has trivial bounded cohomology, therefore the fundamental group $\pi_1(G(i), x)$ has trivial bounded cohomology for all $i \in \I$ and $x$ in $G(i)$. The converse also holds if each space $G(i)$, $i \in \I$, has finitely many components; more generally, the converse holds under the assumption that the components of $G(i)$, for each $i \in \I$, satisfy the uniform boundary condition uniformly (see \cite[Appendix A]{LLM}).  Again this demonstrates the importance of the factorizations of Example \ref{Moore-Postnikov} for the applications of Theorem \ref{nerve_theorem_general2}. We will explore this further in Section \ref{amenable_covers}.
\end{remark}

\subsection{Generalizations} \label{generalizations} Proposition \ref{factorization_general} and Theorems \ref{nerve_theorem_general1} and \ref{nerve_theorem_general2} apply in much greater generality with the same proofs.

Specifically, Proposition \ref{factorization_general} generalizes to natural transformations $c \colon B \to H$ where $B, H \colon \Sp^{\op} \to \D$ are functors to a complete $\infty$-category $\D$ and $H$ is excisive. In this context, $\HR$-colimit diagrams are replaced by diagrams 
$$\widetilde{F} \colon \I \to \Sp_{/X}, i \mapsto (F(i) \to X),$$ such that the canonical map $\colim_{\I} F \to X$ maps to an equivalence in $\D$ under $H$.

Theorem \ref{nerve_theorem_general1}(1) also holds in this generality (with the same proof). In this context, a factorization $\eta \colon \widetilde{F} \to \widetilde{G}$ of $\widetilde{F} \colon \I \to \Sp_{/X}$ is called \emph{acyclic} (with respect to $B$) if $$B(\ast) \xrightarrow{B(G(i) \to \ast)} B(G(i))$$
is an equivalence in $\D$ for every $i \in \I$. Moreover, Theorem \ref{nerve_theorem_general1}(2) and Theorem \ref{nerve_theorem_general2} also apply more generally, e.g., they hold for natural transformations $c \colon B \to H$ as above, where $\D = \ch(R)$ is the derived $\infty$-category of a ring $R$ and the homology of $B(\ast) \in \ch(R)$ is concentrated in degree $0$. In addition, a close inspection of these proofs indicates that it suffices to know only that $H$ acts excisively for the given diagram in the corresponding range of degrees. 

\smallskip 

In particular, Proposition \ref{factorization_general} and Theorems \ref{nerve_theorem_general1} and \ref{nerve_theorem_general2} generalize to analogous results about the comparison map for bounded cohomology with respect to general normed rings or with twisted coefficients (coming from $X$). 

\smallskip

By passing to the opposite $\infty$-categories, we also obtain versions of these results for natural transformations $c \colon H \to B$ where $B, H \colon \Sp \to \D$ are functors to a cocomplete $\infty$-category $\D$ and $H$ is excisive (i.e., $H$ preserves small colimits). In particular, this  yields analogous results to Proposition \ref{factorization_general} and Theorems \ref{nerve_theorem_general1} and \ref{nerve_theorem_general2} for the comparison map for $\ell^1$-homology. In this case, the comparison map corresponds to the assembly map for $\ell^1$-homology (in the sense of Section \ref{sec:prelim}); this should not be confused with the (other) excisive approximation to $\ell^1$-homology (in the sense of Goodwillie calculus) which is shown to be trivial \cite{Wi}. 

\section{Amenable factorizations} \label{amenable_covers}

Amenable covers (or amenable factorizations) are in practice the main source of acyclic or $0$-truncating factorizations (cf. Remarks \ref{acyclic-fact-gromov} and \ref{acyclic-fact-gromov2}). 

\begin{definition}
Let $X \in \Sp$ be a space and let $\widetilde{F} \colon \I \to \Sp_{/X}$ be a diagram. A factorization $\eta \colon \widetilde{F} \to \widetilde{G}$ is called \emph{amenable} if the diagram 
$$\widetilde{G} \colon \I \to \Sp_{/X}, \ i \mapsto (G(i) \to X),$$ satisfies the following property: the fundamental group $\pi_1(G(i), x)$ is amenable  for every $i \in \I$ and $x$ in $G(i)$.
\end{definition}

\begin{corollary}  \label{amenable-fact}
Let $X \in \Sp$ be a space and let $\widetilde{F} \colon \I \to \Sp_{/X}, \ i \mapsto (F(i) \to X),$ be an $\HR$-colimit diagram. Suppose that $\eta \colon \widetilde{F} \to \widetilde{G}$ is an amenable factorization of $\widetilde{F}$, where $\widetilde{G} \colon \I \to \Sp_{/X}, i \mapsto (G(i) \to X)$. Then the following hold: 
\begin{itemize}
\item[(1)]  If the space $G(i)$ is connected for all $i \in \I$, then there is a canonical factorization of $c_X$ in $\ch(\R)$:
$$
\xymatrix{
C^{\bullet}_b(X; \R) \ar[rr]^{c_X} \ar[rd] && C^{\bullet}(X; \R) \\
& C^{\bullet}(|\I|; \R) \ar[ru]_{\rho_F} &
}
$$
where $\rho_F \colon C^{\bullet}(|\I|; \R) \xrightarrow{C^{\bullet}(p_{F}; \R)} C^{\bullet}(\colim_{\I} F; \R) \simeq C^{\bullet}(X; \R).$ 
\item[(2)] If $\I$ is (Joyal) equivalent to a simplicial set of dimension $< m$, then the comparison map $c_X$ vanishes in cohomology in degrees $\geq m$.
\end{itemize}
\end{corollary}
\begin{proof} We claim that the amenable factorization $\eta$ is $0$-truncating.  It follows from the \emph{Mapping Theorem} \cite{Gr, Iv} that the components of $G(i)$ (or the associated fundamental groups) have trivial bounded cohomology for all $i \in \I$. As $G(i)$ may have infinitely many components, in order to show that the bounded cohomology of $G(i)$ itself is concentrated in degree $0$, it is necessary to look at the general properties of the cohomology of bounded products (see \cite[Appendix A]{LLM}). First, there is a uniform bound for the uniform boundary condition of a -- possibly infinite -- family of amenable groups \cite[Example A.10]{LLM}. By the proof of the \emph{Mapping Theorem} (combined with \cite[Proposition A.3]{LLM}), the same property holds also for families of connected spaces with amenable fundamental groups. It follows that the bounded cohomology of $G(i)$, $i \in \I$, is the bounded product of the bounded cohomologies of its components (see \cite[Theorem A.15]{LLM}), therefore $H^{\bullet}_b(G(i); \R)$, $i \in \I$, is concentrated in degree $0$. This proves the claim that amenable factorizations are $0$-truncating. Then (2) is a special case of Theorem \ref{nerve_theorem_general2}(1). Under the additional assumption that $G(i)$ is connected for all $i \in \I$, the $0$-truncating factorization $\eta$ is actually acyclic. Therefore (1) follows from Theorem \ref{nerve_theorem_general1}(1).
\end{proof}

Amenable factorizations often arise in practice from amenable covers of topological spaces (cf. Example \ref{covering-colim}). Let $\U = \{U_i \}_{i \in I}$ be an open cover of a topological space $X$. We recall from Example \ref{covering-colim} that there is an associated colimit diagram 
$$\widetilde{F}_{\U} \colon \I_{\U} \to \Sp_{/X}, \ \sigma \mapsto (U_{\sigma} \xrightarrow{i_{\sigma}} X).$$
As explained in Example \ref{Moore-Postnikov}, the Moore--Postnikov truncation (at $\pi_1$) of the maps $(i_{\sigma} \colon U_{\sigma} \rightarrow X)$ defines a new diagram 
$$\mathcal{T}(\widetilde{F}_{\U}) \colon \I_{\U} \to \Sp_{/X}, \ \sigma \mapsto (V_{\sigma} := \mathcal{T}(i_{\sigma}) \xrightarrow{j_{\sigma}} X)$$
together with a factorization $p \colon \widetilde{F}_{\U} \to \ImFU$.
We recover the statements in Theorem \ref{nerve_theorem} as special cases:

\smallskip

\noindent (1) \emph{Covering theorem}. Suppose that $p \colon \widetilde{F}_{\U} \to \mathcal{T}(\widetilde{F}_{\U})$ is an acyclic factorization of $\widetilde{F}_{\U}$, i.e., the space $V_{\sigma}$ have trivial bounded cohomology for every $\sigma \in \I_{\U}$. Then there is a canonical factorization in $\ch(\R)$ (Theorem \ref{nerve_theorem_general1}):  
$$
\xymatrix{
C^{\bullet}_b(X; \R) \ar[rr]^{c_X} \ar[rd] && C^{\bullet}(X; \R). \\
& C^{\bullet}(|\I_{\U}|; \R) \ar[ur]_{\rho_{F_{\U}}} & 
}
$$

\smallskip

\noindent The factorization $p$ is acyclic in the case where $U_{\sigma} \subseteq X$ is path-connected for every $\sigma \in \I_{\U}$ and the image of the homomorphism $(\pi_1(U_i, x) \to \pi_1(X, x))$ is amenable for every $i \in I$. In this case, for every $\sigma \in \I_{\U}$, the space $V_{\sigma}$ is connected in $\Sp$ with amenable fundamental group -- this uses the fact that subgroups of amenable groups are again amenable. In other words, the factorization $p \colon \widetilde{F}_{\U} \to \mathcal{T}(\widetilde{F}_{\U})$ is amenable and satisfies the assumptions of Corollary \ref{amenable-fact}(1). This case recovers Theorem \ref{nerve_theorem}(1).

\smallskip

\noindent We also refer to the recent work of Ivanov \cite{Iv2} for closely related results.  

\smallskip

\noindent (2) \emph{Vanishing theorem}. Suppose that $p \colon \widetilde{F}_{\U} \to \mathcal{T}(\widetilde{F}_{\U})$ is  a $0$-truncating factorization of $\widetilde{F}_{\U}$. If the dimension of $\I_{\U}$ is $< m$ (i.e., if $U_{\sigma} = \varnothing$ when $\sigma$ has $m+1$ elements), then the comparison map
$$c^n_X \colon H^n_b(X; \R) \to H^n(X; \R)$$
is zero for $n \geq m$ (Theorem \ref{nerve_theorem_general2}). 

\smallskip

\noindent The factorization $p \colon \widetilde{F}_{\U} \to \mathcal{T}(\widetilde{F}_{\U})$ is $0$-truncating when $p$ is  an amenable factorization, that is, when the image of the homomorphism $(\pi_1(U_i, x) \to \pi_1(X, x))$ is amenable for every $x \in U_i$ and $i \in I$ (see the proof of Corollary \ref{amenable-fact}). We recall that the open cover $\U$ is called \emph{amenable} in this case -- note that $U_i$ are not required to be path-connected. This case recovers Theorem \ref{nerve_theorem}(2).

\smallskip

More generally, 

\smallskip

\noindent (3) Suppose that $p \colon \widetilde{F}_{\U} \to \mathcal{T}(\widetilde{F}_{\U})$ is  a $k$-truncating factorization of $\widetilde{F}_{\U}$. In this case, we call the open cover $\U$ $k$-\emph{truncating}. If $\I_{\U}$ has dimension $d$, then the comparison map 
$$c_X^n \colon H^n_b(X; \R) \to H^n(X; \R)$$
is zero for $n > d + k$ (Theorem \ref{nerve_theorem_general2}).  On the other hand, if $p \colon \widetilde{F}_{\U} \to \mathcal{T}(\widetilde{F}_{\U})$ is a $k$-acyclic factorization, then the comparison map $c_X^n$ factors through the map $\rho_{F_{\U}}^n \colon H^n(|\I_{\U}|; \R) \to H^n(X; \R)$ for $n \leq k$ (Theorem \ref{nerve_theorem_general1}).

\medskip

We note that these results (esp. the vanishing result in (2)) are useful for showing the vanishing of the simplicial volume of $X$, when $X$ is an oriented closed connected manifold; for example, see \cite[3.1]{Gr}.

\begin{application}[Families of amenable covers] This example of application is a slight variation of a recent result of L\"oh--Moraschini \cite[Corollary 1.2]{LM}. Let $p \colon E \to B$ be a fiber bundle between path-connected spaces and let $\U = \{U_i\}_{i \in I}$ be an open cover of $B$ by trivializing neighborhoods for $p$. In addition, suppose that the path-components of $U_{\sigma}$ have amenable fundamental groups for every $\sigma \in \I_{\U}$. Let $d$ denote the dimension of $\I_{\U}$. If $B$ is a finite simplicial complex (or a smooth closed manifold), then there is such an open cover of $B$ with $d = \mathrm{dim}(B)$. 

Let $F$ denote the fiber of $p$ (at a basepoint $b \in B$) and suppose that $\V = \{V_j\}_{j \in J}$ is an amenable (or $k$-truncating) open cover of $F$. We denote by $\kappa$ the dimension of $\I_{\V}$. Let $\U' = \{U'_i = p^{-1}(U_i)\}_{i \in I}$ denote the induced open cover of $E$. Using an identification $U'_i \cong U_i \times F$, the amenable open cover $\{U_i \times V_j\}_{j \in J}$ of $U_i \times F$ determines an amenable open cover of $U'_i$, denoted by $\V_i = \{V_{ij}\}_{j \in J}$  -- this is amenable because amenable groups are closed under products. 

These choices of open covers for $\{U'_i\}_{i \in I}$ determine an amenable open cover $\U \times \V:= \bigcup_{i \in I} \V_i$ of $E$ -- this is amenable because amenable groups are closed under quotients. An elementary combinatorial argument shows that the dimension of $\I_{\U \times \V}$ is less than $(d+1)(\kappa+1)$.
Therefore the comparison map 
$$c^n_E: H^n_b(E; \R) \to H^n(E;\R)$$ 
vanishes for $n \geq (d+1)(\kappa + 1)$. Assuming $E$ is an oriented closed connected manifold, this vanishing result implies the vanishing of the simplicial volume of $E$ when $\mathrm{dim}(E) \geq (d+1)(\kappa + 1)$.  
\end{application}

\begin{example}[The plus construction] Let $X$ be a based connected space (for example, $X = BG$ for a discrete group $G$) and 
let $X^+_P$ denote the plus construction of $X$ associated to a normal perfect subgroup $P$ of $\pi_1(X, *)$ \cite{HH, Ra}. We recall that every connected space is equivalent to $BG^+_P$ for some discrete group $G$ and a normal perfect subgroup $P \unlhd G$ \cite{KT}. 

Then every colimit diagram $\widetilde{F} \colon \I^{\triangleright} \to \Sp_{/X}$ gives rise to an $\HR$-colimit diagram 
$\iota_*\widetilde{F} \colon \I^{\triangleright} \to \Sp_{/X^+_P}$ by composition with the canonical acyclic map 
$\iota \colon X \to X^+_P$. Moreover, a ($k$-acyclic/truncating or amenable) factorization of $\widetilde{F}$ defines also a ($k$-acyclic/truncating or amenable) factorization of $\iota_*\widetilde{F}$. 

Conversely, given a factorization $\eta \colon \iota_*\widetilde{F} \to \widetilde{G}$ of $\iota_*\widetilde{F}$, there is an associated factorization of $\widetilde{F},$
$$\iota^*\eta \colon \widetilde{F} \to \iota^* \widetilde{G},$$ 
where $\iota^*\widetilde{G} \colon \I \to \Sp_{/X}, \ i \mapsto (G(i) \times_{X^+_P} X \to X),$ is defined by pulling back $\widetilde{G}$ along the map $\iota$; in addition, there are natural acyclic maps $G(i) \times_{X^+_P} X \to G(i)$ for every $i \in \I$. Thus, applying Proposition \ref{factorization_general}, we obtain a factorization of $c_X$ in $\ch(\R)$:
$$
\xymatrix{
C^{\bullet}_b(X; \R) \ar[r]^{c_X} \ar[d] & C^{\bullet}(X; \R).  \\
\lim_{\I^{\op}}C^{\bullet}_b(G(-) \times_{X^+_P} X; \R) \ar[r] & \lim_{\I^{\op}} C^{\bullet}(G(-) \times_{X^+_P} X; \R) \ar[u]  \simeq \lim_{\I^{\op}} C^{\bullet}(G(-); \R) 
}
$$

In this way, we may obtain factorizations of the comparison map $c_X$ of $X$ using factorizations of the comparison map $c_{X^+_P}$ of $X^+_P$ whose fundamental group may be simpler in general. 
\end{example}

\section{Parametrized factorizations} \label{pullback_fact}

\subsection{General parametrized factorizations} \label{general_parametrized} For any map $p \colon X \to B$ in $\Sp$, we may carry out a similar analysis of the properties of the comparison map $c_X \colon C^{\bullet}_{b}(X;\R) \to C^{\bullet}(X;\R)$ in a parametrized setting. In this section we briefly sketch some of the details and discuss some specific examples and applications. 

\smallskip

Let $\Sp^{\to} := \mathrm{Fun}(\Delta^1, \Sp)$ denote the $\infty$-category of maps in $\Sp$. Note that a diagram $$\widetilde{F} = (\widetilde{F}_0 \to \widetilde{F}_1) \colon \I \to (\Sp^{\to})_{/p}$$ corresponds to a cone $\widetilde{F} \colon \I^{\triangleright} \to \Sp^{\to}$ on $F:=\widetilde{F}_{| \I} \colon \I \to \Sp^{\to}$, written $i \mapsto F(i) = (F_0(i) \to F_1(i))$, with cone object $p \in \Sp^{\to}$. Such a diagram in $(\Sp^{\to})_{/p}$ is a \emph{colimit diagram} (resp. \emph{$\HR$-colimit diagram}) if the associated cone is a (pointwise) colimit diagram (resp. $\HR$-colimit diagram) in $\Sp^{\to}$. 

A diagram $\widetilde{G} = (\widetilde{G}_0 \to \widetilde{G}_1) \colon \I \to (\Sp^{\to})_{/p}$, $i \mapsto \big(G(i) = (G_0(i) \to G_1(i)) \to p\big),$ together with a natural transformation $$\eta \colon \widetilde{F} \to\widetilde{G}$$ is called a \emph{factorization} of $\widetilde{F} \colon \I \to (\Sp^{\to})_{/p}$. Explicitly, a factorization $\eta \colon \widetilde{F} \to \widetilde{G}$ of $\widetilde{F}$ consists of a factorization $\eta_0 \colon \widetilde{F}_0 \to \widetilde{G}_0$ of $\widetilde{F}_0 \colon \I \to \Sp_{/X}$ and a factorization $\eta_1 \colon \widetilde{F}_1 \to \widetilde{G}_1$ of $\widetilde{F}_1 \colon \I \to \Sp_{/B}$ which fit into a natural diagram for all $i \in \I$:
$$
\xymatrix{
F_0(i) \ar[d] \ar[r]^{\eta_{0, i}} & G_0(i) \ar[d] \ar[r] & X \ar[d]^p \\
F_1(i) \ar[r]^{\eta_{1,i}} & G_1(i) \ar[r] & B.
}
$$
 
\begin{example}[Pullback factorizations] Let $\widetilde{F}_1 \colon \I \to \Sp_{/B}, \ i \mapsto (F_1(i) \to B),$ be a (colimit) diagram and $\eta_1 \colon \widetilde{F}_1 \to \widetilde{G}_1$ a factorization of $\widetilde{F}_1$, where $\widetilde{G}_1 \colon \I \to \Sp_{/B}, \ i \mapsto (G_1(i) \to B)$. These diagrams induce by pullback along $p$ a (colimit) diagram (because pulling back commutes with colimits in $\Sp$)
$$\widetilde{F}_0 \colon \I \to \Sp_{/X}, \ i \mapsto (F_0(i) = F_1(i) \times_B X \to X),$$
and a factorization $\eta_0 \colon \widetilde{F}_0 \to \widetilde{G}_0$, where  
$$\widetilde{G}_0 \colon \I \to \Sp_{/X}, \ i \mapsto (G_0(i) = G_1(i) \times_B X \to X).$$
Moreover, the canonical maps $F_0(i) \to F_1(i)$ and $G_0(i) \to G_1(i)$ determine diagrams $\widetilde{F} = (\widetilde{F}_0 \to \widetilde{F}_1) \colon \I \to (\Sp^{\to})_{/p}$ and $\widetilde{G} = (\widetilde{G}_0 \to \widetilde{G}_1) \colon \I \to (\Sp^{\to})_{/p}$ together with a factorization $\eta = (\eta_0 \to \eta_1)\colon \widetilde{F} \to \widetilde{G}$. 
\end{example}

The use of (parametrized) factorizations for the study of the comparison map is based on the diagram below (cf. Proposition \ref{factorization_general}). Let $\widetilde{F} = (\widetilde{F}_0 \to \widetilde{F}_1) \colon \I \to (\Sp^{\to})_{/p}$ be a diagram such that $\widetilde{F}_0 \colon \I \to \Sp_{/X}$ is an $\HR$-colimit diagram. Given another diagram $\widetilde{G} = (\widetilde{G}_0 \to \widetilde{G}_1) \colon \I \to (\Sp^{\to})_{/p}$ and a factorization $\eta \colon \widetilde{F} \to \widetilde{G}$ of $\widetilde{F}$, then we obtain the following diagram in $\D(\R)$:
\begin{equation*} \label{fact_param} 
\xymatrix{
C^{\bullet}_b(X; \R) \ar[r]^{c_X} \ar[d] & C^{\bullet}(X; \R). \ar[d] \ar[dr]^{\simeq} & \\
\lim_{\I^{\op}} C^{\bullet}_b(G_0(-); \R) \ar[r]^{[**]} & \lim_{\I^{\op}} C^{\bullet}(G_0(-); \R)  \ar[r] & \lim_{\I^{\op}} C^{\bullet}(F_0(-); \R) \\
\lim_{\I^{\op}} C^{\bullet}_b(G_1(-); \R) \ar[u]^{[*]} \ar[r] & \lim_{\I^{\op}} C^{\bullet}(G_1(-); \R) \ar[u]  \ar[r] & \lim_{\I^{\op}} C^{\bullet}(F_1(-); \R) \ar[u] \\
}
\end{equation*}
Moreover, if $\widetilde{F}_1 \colon \I \to \Sp_{/B}$ is also an $\HR$-colimit diagram, then we have a canonical equivalence $\lim_{\I^{\op}} C^{\bullet}(F_1(-); \R) \simeq C^{\bullet}(B; \R)$ such that the right vertical composition in the diagram is identified with the map $C^{\bullet}(p;\R) \colon C^{\bullet}(B;\R) \to C^{\bullet}(X;\R)$.

\smallskip

Special properties of $\eta$ or $p$ would generally allow us to conclude useful properties of the maps indicated in the diagram by [$\ast$] and [$\ast\ast$], and therefore obtain information about the comparison map $c_X$. In analogy with Theorems \ref{nerve_theorem_general1} and \ref{nerve_theorem_general2}, the goal of this analysis is either to relate the comparison map $c_X$ with the bounded cohomology of $B$ or show its vanishing in a range of degrees.  In order to view Theorems \ref{nerve_theorem_general1} and \ref{nerve_theorem_general2} as special cases of the present context applied to the map $p \colon X \to \ast$, it suffices to note that $\Sp_{/X}$ can be identified with the full subcategory of $(\Sp^{\to})_{/p}$, which is spanned by objects of the form
$$
\xymatrix{
A \ar[r] \ar[d] & X \ar[d] \\
\ast \ar@{=}[r] & \ast.}
$$ 
In this way we may view diagrams and their factorizations in $\Sp_{/X}$ as diagrams and factorizations in $(\Sp^{\to})_{/p}$. (The present parametrized setting can also be extended to objects $p \colon K \to \Sp$ and diagrams and their factorizations in $\mathrm{Fun}(K, \Sp)_{/p}$ for a based simplicial set $K$.)

\begin{example}\label{pullback_fact_1}
In the notation introduced above, suppose that the map $G(i)$ induces isomorphisms in bounded cohomology for all $i \in \I$, that is, the induced map $H^n_b(G_1(i);\R) \to H^n_b(G_0(i);\R)$ is an isomorphism for every $n \geq 0$.  In this case, the map [$\ast$] in the diagram is an equivalence in $\D(\R)$, so we obtain a canonical factorization:
\begin{equation} \label{fact_5.2} \tag{\ref{pullback_fact_1}}
\xymatrix{
C^{\bullet}_b(X; \R) \ar[rr]^{c_X} \ar[dr] && C^{\bullet}(X; \R). \\
&  \lim_{\I^{\op}} C^{\bullet}(F_1(-); \R) \ar[ur] &
}
\end{equation}
If $\widetilde{F}_1 \colon \I \to \Sp_{/B}$ is an $\HR$-colimit diagram, then this factorization \eqref{fact_5.2} is identified with the factorization:
\begin{equation*}
\xymatrix{
C^{\bullet}_b(X; \R) \ar[rr]^{c_X} \ar[dr] && C^{\bullet}(X; \R). \\
&  C^{\bullet}(B; \R) \ar[ru]_{C^{\bullet}(p; \R)} \ar[ur] &
}
\end{equation*}
On the other hand, if $B=\ast$ and $\widetilde{F}_1 \colon \I \to \Sp$ is the constant diagram $\delta(\ast)$, then the factorization \eqref{fact_5.2} is identified with the factorization of Theorem \ref{nerve_theorem_general1}(1). So we may view the factorization \eqref{fact_5.2} as a relative version of Theorem \ref{nerve_theorem_general1}(1). 

Under appropriate assumptions on the maps $G(i)$, $i \in \I$, there are similar versions of Theorem \ref{nerve_theorem_general1}(2) and Theorem \ref{nerve_theorem_general2} in this context, too. We leave the details to the interested reader. 
\end{example}

\subsection{Pullbacks of open covers} \label{pullback_covers} We briefly discuss the special case of (possibly only $1$-categorical) pullback factorizations arising from open covers (Example \ref{covering-colim}). 

\smallskip

Let $p \colon E \to B$ be a surjective map of topological spaces and let $\U = \{U_i\}_{i \in I}$ be an open cover of $B$.  Let $\U' = \{U'_i = p^{-1}(U_i)\}_{i \in I}$ denote the induced open cover of $E$ -- note that $\I_{\U'} = \I_{\U}$. An important example is when $p: E \to B$ is a fiber bundle between oriented closed connected manifolds and $\U$ is an open cover of $B$ by trivializing neighborhoods for $p$. In this case, the factorizations of this subsection are often useful for showing the vanishing of the simplicial volume of $E$.

The open covers $\U$ and $\U'$ yield a diagram $\widetilde{F}_{\U/p} \colon \I_{\U} \to (\Sp^{\to})_{/p}$ which is defined for each $\sigma \in \I_{\U}$ by the square of topological spaces
$$
\xymatrix{
U'_{\sigma} = p^{-1}(U_{\sigma}) \ar[d]_{p_{\sigma}} \ar[r] & E \ar[d]^p \\
U_{\sigma} \ar[r] & B.
}
$$

\smallskip

\noindent \emph{(a) $\U$-local bounded cohomology equivalence} (cf. Example \ref{pullback_fact_1}). Suppose that the following property is satisfied: the map induced by $p_{\sigma} = p_{| U'_{\sigma}} \colon U'_{\sigma} \to U_{\sigma}$,
$$C^{\bullet}_b(U_{\sigma}; \R) \to C^{\bullet}_b(U'_{\sigma}; \R),$$
is an equivalence in $\ch(\R)$ for every $\sigma \in \I_{\U}$. This happens, for example, when $U_{\sigma}$ is weakly contractible and $U'_{\sigma}$ is path-connected with amenable (or boundedly acyclic) fundamental group for every $\sigma \in \I_{\U}$. 

Then, using the tautological factorization (Example \ref{ex3}) of $\widetilde{F}_{\U/p}$ and the equivalence $\lim_{\sigma \in\I^{\op}_{\U}} C^{\bullet}(U_{\sigma}; \R) \simeq C^{\bullet}(B; \R)$, we obtain a canonical factorization in $\ch(\R)$:
\begin{equation*}
\xymatrix{
C^{\bullet}_b(E; \R) \ar[rr]^{c_E} \ar[dr] && C^{\bullet}(E; \R). \\
& C^{\bullet}(B; \R) \ar[ur]_{C^{\bullet}(p; \R)} 
}
\end{equation*}

\medskip

\noindent \emph{(b) $\U$-local boundedly equivalent factorizations} (cf. Example \ref{pullback_fact_1}). 
Let $V_{\sigma}$ (resp. $V'_{\sigma}$) be the Moore--Postnikov truncation of the inclusion map $U_{\sigma} \to B$ (resp. $U'_{\sigma} \to E$) as in Example \ref{Moore-Postnikov} and Section \ref{amenable_covers}. There are natural maps $p'_{\sigma} \colon V'_{\sigma} \to V_{\sigma}$ by the functoriality of the Moore--Postnikov truncation, yielding a parametrized factorization $\eta \colon \widetilde{F}_{\U/p} \to \widetilde{G}$; explicitly, this is given for each $\sigma \in \I_{\U}$ by the diagram
$$
\xymatrix{
U'_{\sigma} \ar[r] \ar[d]_{p_{\sigma}} & V'_{\sigma} \ar[d]_{p'_{\sigma}} \ar[r] & E \ar[d]^p \\
U_{\sigma} \ar[r] & V_{\sigma} \ar[r] & B.
}
$$
Suppose that the following property is satisfied: the map induced by $p'_{\sigma} \colon V'_{\sigma} \to V_{\sigma},$ 
$$C^{\bullet}_b(V_{\sigma}; \R) \to C^{\bullet}_b(V'_{\sigma}; \R),$$
is an equivalence in $\ch(\R)$ for every $\sigma \in \I_{\U}$. For example, this property is satisfied when the spaces $U_{\sigma}, U'_{\sigma}$ are path-connected for every $\sigma \in \I_{\U}$ and $\pi_1(p'_{\sigma})$ is surjective with boundedly acyclic kernel (see \cite{MR}). Then, proceeding as in Example \ref{pullback_fact_1}, we obtain a canonical factorization in $\ch(\R)$:
\begin{equation*} 
\xymatrix{
C^{\bullet}_b(E; \R) \ar[rr]^{c_E} \ar[dr] && C^{\bullet}(E; \R). \\
& C^{\bullet}(B; \R) \ar[ur]_{C^{\bullet}(p; \R)} & 
}
\end{equation*}

\subsection{Relative covering/vanishing theorems} \label{subsec:rel-cov-van} For an inclusion of topological spaces $i \colon A \subseteq X$, the \emph{relative bounded cohomology $C^{\bullet}_b(X, A; \R) \in \ch(\R)$} of the pair $(X,A)$ is defined by 
$$C^{\bullet}_b(X, A;\R) := \mathrm{fib}\big(C^{\bullet}_b(X;\R) \xrightarrow{C^{\bullet}_b(i; \R)} C^{\bullet}_b(A;\R)\big).$$
Moreover, we may define similarly the relative bounded cohomology of an arbitrary map in $\Sp$. Similarly to the absolute case, there is also an induced natural comparison map in $\ch(\R)$ in the relative setting 
$$c_{X,A} \colon C^{\bullet}_b(X,A;\R) \to C^{\bullet}(X,A;\R).$$ 
The analysis of parametrized factorizations in Subsections \ref{general_parametrized} and \ref{pullback_covers} can be adjusted to the relative setting simply by passing to the respective fibers -- this is another instance where working directly in $\D(\R)$ can be useful. 

More precisely, given an $\HR$-colimit cone $\widetilde{F} \colon \I^{\triangleright} \to (\Sp^{\to})_{/i}$ together with a factorization $\eta \colon \widetilde{F} \to \widetilde{G}$ (as in Subsection \ref{general_parametrized}), we obtain a factorization of $c_{X,A}$ in $\ch(\R)$:
$$\small{
\xymatrix{
C^{\bullet}_b(X, A; \R) \ar[r]^{c_{X,A}} \ar[d] & C^{\bullet}(X, A; \R). \ar[d]  \ar@/^6pc/[dd]^{\simeq} \\
\lim_{\I^{\op}} C^{\bullet}_b(G_1(-), G_0(-); \R) \ar[r] & \lim_{\I^{\op}} C^{\bullet}(G_1(-), G_0(-); \R)  \ar[d] \\
&  \lim_{\I^{\op}} C^{\bullet}(F_1(-), F_0(-); \R) 
}
}
$$
Similarly to the absolute case, we may identify conditions on the factorization $\eta$ such that the relative bounded cohomology groups $H^{\bullet}_b(G_1(i), G_0(i);\R)$ vanish in certain degrees, and thus obtain a description of the limit in $\ch(\R)$
$$\lim_{\I^{\op}} C^{\bullet}_b(G_1(-), G_0(-);\R)$$ 
(as in Theorem \ref{nerve_theorem_general1} and Example \ref{pullback_fact_1}) or conclude the vanishing of $c_{X,A}$ in certain degrees (as in Theorem \ref{nerve_theorem_general2}). 

The possible examples are numerous; we will illustrate the application of this general method with the following example, which shows versions of the covering and vanishing theorems in the relative context. Similar results (with different sets of assumptions) were also obtained recently by Li--L\"oh--Moraschini \cite{LLM}.  

\begin{application}[Relative covering and vanishing theorems via open covers] \label{relative-open-covers}
Let $\U = \{U_i\}_{i \in I}$ be an amenable (or $0$-truncating) open cover of $X$ and let $\U_A = \{U_i \cap A\}_{i \in I}$ denote the induced open cover of $A$. Together they determine a colimit cone $\widetilde{F}_{\U, \U_A} \colon \I_{\U}^{\triangleright} \to (\Sp^{\to})_{/i}$, which sends $\sigma \in \I_{\U}$ to the square
$$
\xymatrix{
U_{\sigma} \cap A \ar[r] \ar[d] & A \ar[d]^i \\
U_{\sigma} \ar[r] & X.
}
$$
Note that $U_{\sigma} \cap A$ may be empty for $\sigma \in \I_{\U}$, so $\I_{\U_A}$ may be strictly smaller than $ \I_{\U}$; $\I_{\U_A}$ is upward closed in $\I_{\U}$ (viewed as posets). Let $d$ denote the dimension of (the nerve of) $\I_{\U}$ and $d_A$ the dimension of $\I_{\U_A}$. 

Assume that the open cover $\U_A$ of $A$ is again amenable (or $0$-truncating). For example, this holds if the inclusion $i \colon A \subseteq X$ is $\pi_1$-injective (for all basepoints) because amenability is preserved under passing to subgroups. 

\smallskip

The Moore--Postnikov truncation (at $\pi_1$), applied seperately to the domain and target of an object in $\Sp^{\to}$, determines a factorization $\widetilde{F}_{\U, \U_A} \to \widetilde{G}:= \mathcal{T}(\widetilde{F}_{\U, \U_A})$. Using our assumptions, it follows that the cochain complexes
$C^{\bullet}_b(G_1(\sigma);\R)$ and $C^{\bullet}_b(G_0(\sigma);\R)$ are concentrated in degree $0$. As a consequence, the limit in $\ch(\R)$ 
$$\lim_{\I_{\U}^{\op}} C^{\bullet}_b(G_1(-); \R)$$ is concentrated in cohomological degrees $\leq d$. On the other hand, we observe that $C^{\bullet}_b(G_0(\sigma);\R) \simeq 0$ if $\sigma \notin \I_{\U_A}$, so the canonical map
$$\lim_{\I_{\U}^{\op}} C^{\bullet}_b(G_0(-); \R) \to \lim_{\I_{\U_A}^{\op}} C^{\bullet}_b(G_0(-); \R)$$
is an equivalence in $\ch(\R)$ and the latter cochain complex is concentrated in cohomological degrees $\leq d_A$. Hence, by passing to the fibers, the relative comparison map $c_{X,A}$ factors through the cochain complex 
$$\mathrm{fib}\big(\lim_{\I_{\U}^{\op}} C^{\bullet}_b(G_1(-);\R) \longrightarrow \lim_{\I_{\U_A}^{\op}} C^{\bullet}_b(G_0(-);\R)\big)$$
whose cohomology is concentrated in degrees $\leq \mathrm{max}\{d, d_A +1\}$. This implies that the relative comparison map (\emph{relative vanishing theorem})
\begin{equation} \tag{5.3(i)}
c^n_{X,A} \colon H^n_b(X, A; \R) \to H^n(X, A;\R)
\end{equation}
vanishes for $n > \mathrm{max}\{d, d_A + 1\}$. Applied to the simplicial volume of an oriented compact connected $n$-manifold $(M, \partial M)$ 
(see also \cite[Proposition 2.18]{LMR}), this vanishing result recovers the relative vanishing theorem in \cite[Theorem 3.13]{LMR} as a special case. 

\smallskip

\noindent In addition, noting that the relative comparison map $c_{X, A}$ factors through the cochain complex 
\begin{equation} \tag{5.3(ii)}
\lim_{\I_{\U}^{\op}} C^{\bullet}_b(G_1(-), G_0(-);\R),
\end{equation}
we observe that for $\sigma \in \I_{\U}$:
\begin{itemize}
\item[(a)] $C^{\bullet}_b(G_1(\sigma), G_0(\sigma);\R)$ is concentrated in cohomological degrees $\leq 1$;
\item[(b)]  $C^{\bullet}_b(G_1(\sigma), G_0(\sigma);\R)$ is concentrated in degree $0$ if and only if the map $\pi_0(U_{\sigma} \cap A) \to \pi_0(U_{\sigma})$ is injective;
\item[(c)]  $C^{\bullet}_b(G_1(\sigma), G_0(\sigma);\R) \simeq \R$, if $U_{\sigma}$ is path-connected and $U_{\sigma} \cap A$ is  empty; 
\item[(d)] $C^{\bullet}_b(G_1(\sigma), G_0(\sigma);\R) \simeq 0$ if and only if the map $\pi_0(U_{\sigma} \cap A) \to \pi_0(U_{\sigma})$ is a bijection. 
\end{itemize}
As a consequence, if each $U_{\sigma}$ or $U_{\sigma} \cap A$ is either empty or path-connected for all finite $\sigma \subseteq I$, then the relative comparison map $c_{X,A}$ factors canonically through 
\begin{align*}
\lim_{\I_{\U}^{\op}} C^{\bullet}_b(G_1(-), G_0(-); \R) & \simeq \mathrm{fib}\big(\lim_{\I_{\U}^{\op}} C^{\bullet}_b(G_1(-); \R) \to \lim_{\I_{\U_{A}}^{\op}}C^{\bullet}_b(G_0(-); \R) \big) \\ & \simeq \mathrm{fib}\big(\lim_{\I_{\U}^{\op}} C^{\bullet}_b(\ast; \R) \to \lim_{\I_{\U_{A}}^{\op}}C^{\bullet}_b(\ast; \R) \big) \\ & \simeq \mathrm{fib}\big(\lim_{\I_{\U}^{\op}} C^{\bullet}(\ast; \R) \to \lim_{\I_{\U_{A}}^{\op}}C^{\bullet}(\ast; \R) \big) \\ & \simeq \mathrm{fib}\big(C^{\bullet}(|\I_{\U}|; \R) \to C^{\bullet}(|\I_{\U_{A}}|; \R) \big) \\ & \simeq C^{\bullet}(|\I_{\U}|, |\I_{\U_A}|; \R).
\end{align*}
More precisely, a close inspection of these identifications (cf. the proof of Theorem \ref{nerve_theorem_general1}) shows that $c_{X,A}$ admits a canonical factorization in $\ch(\R)$ (\emph{relative covering theorem}):
\begin{equation} \tag{5.3(iii)}
\xymatrix{
C^{\bullet}_b(X,A; \R) \ar[rr]^{c_{X,A}} \ar[dr] && C^{\bullet}(X,A; \R) \\
& C^{\bullet}(|\I_{\U}|, |\I_{\U_A}|; \R) \ar[ur]_{\ \ \rho_{F_{\U}, F_{\U_A}}} &
}
\end{equation}
where $\rho_{F_{\U}, F_{\U_A}}$ is induced by the corresponding maps $\rho_{F_{\U}}$ and $\rho_{F_{\U_A}}$ (as defined in Theorem \ref{nerve_theorem_general1} and Section \ref{amenable_covers}).

\smallskip

\noindent On the other hand, if (b) is satisfied for all $\sigma \in \I_{\U}$ and the subposet $$\mathcal{J} = \{\sigma \in \I_{\U} \ | \ \sigma \text{ does not satisfy } (d)\} \subseteq \I_{\U}$$ is upward closed of dimension (of its nerve) denoted by $d_{\mathcal{J}}$, then we have a canonical equivalence in $\ch(\R)$:
\begin{equation} \tag{5.3(iv)}
\lim_{\I_{\U}^{\op}} C^{\bullet}_b(G_1(-), G_0(-); \R) \xrightarrow{\simeq} \lim_{\mathcal{J}^{\op}} C^{\bullet}_b(G_1(-), G_0(-); \R)
\end{equation}
and therefore the comparison map $c_{X,A}$ factors through a cochain complex concentrated in cohomological degrees $\leq d_{\mathcal{J}}$. In this case, the relative comparison map 
\begin{equation} \tag{5.3(v)}
c^n_{X,A} \colon H^n_b(X, A; \R) \to H^n(X, A;\R)
\end{equation} 
vanishes for $n > d_{\mathcal{J}}$. 
\end{application}


\begin{thebibliography}{10}

\bibitem{Ci}
D.-C.~Cisinski, \emph{Higher categories and homotopical algebra}. Cambridge Studies in Advanced Mathematics Vol.~180, Cambridge University Press, Cambridge, 2019.

\bibitem{Du}
D.~Dugger, \emph{A primer on homotopy colimits.} Expository project (2008).  Updated version (2017) available online at:  \url{https://pages.uoregon.edu/ddugger/hocolim.pdf}   

\bibitem{DI}
D.~Dugger and D.~C.~Isaksen, \emph{Topological hypercovers and $\mathbb{A}^1$-realizations.} Math. Z. 246 (2004), 667--689.

\bibitem{DHKS}
W.~G.~Dwyer, P.~S.~Hirschhorn, D.~M.~Kan, and J.~H.~Smith, \emph{Homotopy limit functors on model categories and homotopical categories.} Mathematical Surveys and Monographs Vol.~113, American Mathematical Society, Providence, RI, 2004.

\bibitem{FM}
R.~Frigerio and M.~ Moraschini, \emph{Gromov's theory of multicomplexes with applications to bounded cohomology and simplicial volume}. Mem. Amer. Math. Soc. (to appear). arXiv: \url{https://arxiv.org/abs/1808.07307}

\bibitem{GJ}
P.~G.~Goerss and J.~F.~Jardine, \emph{Simplicial homotopy theory.} Reprint of the 1999 edition. Modern Birkh\"{a}user Classics, Birkh\"{a}user Verlag, Basel, 2009.

\bibitem{Gr}
M.~Gromov, \emph{Volume and bounded cohomology.} Publ. Math. IH\'{E}S 56 (1982), 5--99.

\bibitem{HH}
J.-C.~Hausmann and D.~Husemoller, \emph{Acyclic maps}. Enseign. Math. (2) 25 (1979), no.~1--2, 53--75.

	
\bibitem{Hi}
V.~Hinich, \emph{Dwyer-{K}an localization revisited}. Homology Homotopy Appl. 18 (2016), no.~1, 27--48.

\bibitem{Iv}
N.~V.~Ivanov, \emph{Notes on the bounded cohomology theory.} Preprint (2017). arXiv: \url{https://arxiv.org/abs/1708.05150}

\bibitem{Iv2}
N.~V.~Ivanov, \emph{Leray theorems in bounded cohomology theory.} Preprint (2020). arXiv: \url{https://arxiv.org/abs/2012.08038}


\bibitem{KT}
D.~M.~Kan and W.~P.~Thurston, \emph{Every connected space has the homology of a {$K(\pi ,1)$}}. Topology 15 (1976), no.~3, 253--258.

\bibitem{LLM}
K.~Li, C.~L\"oh and M.~Moraschini, \emph{Bounded acyclicity and relative simplicial volume}. Preprint (2022). arXiv: \url{https://arxiv.org/abs/2202.05606}

\bibitem{LM}
C.~L\"oh and M.~Moraschini, \emph{Topological volumes of fibrations: A note on open covers}. Proc. R. Soc. Edinb. A: Math. 152 (2022), no.~5, 1340--1360. 

\bibitem{LMR}
C.~L\"oh, M.~Moraschini and G.~Raptis, \emph{On the simplicial volume and the Euler characteristic of (aspherical) manifolds}. Res. Math. Sci. 9:44 (2022). 

\bibitem{LS}
C.~L\"oh and R.~Sauer, \emph{Bounded cohomology of amenable covers via classifying spaces}. Enseign. Math. 66 (2020), no. 1--2, 151--172.

\bibitem{HTT}
J.~Lurie, \emph{Higher topos theory.} Annals of Mathematics Studies Vol. 170, Princeton University Press, Princeton, NJ, 2009. Revised version available online at: \url{https://www.math.ias.edu/~lurie/papers/HTT.pdf}

\bibitem{MR}
M.~Moraschini and G.~Raptis, \emph{Amenability and acyclicity in bounded cohomology theory}. Preprint (2021). arXiv: \url{https://arxiv.org/abs/2105.02821}
 
\bibitem{Ra}
G.~Raptis,  \emph{Some characterizations of acyclic maps}. J. Homotopy Relat. Struct. 14 (2019), no.~3, 773--785.

\bibitem{WW}
M.~Weiss and B.~Williams, \emph{Assembly.} In: \emph{Novikov conjectures, index theorems and rigidity}, {V}ol. 2 ({O}berwolfach, 1993). London Math. Soc. Lecture Note Ser. 227 (1995), pp. 332--352, Cambridge University Press, Cambridge.

\bibitem{Wh}
G.~W.~Whitehead, \emph{Elements of homotopy theory.} Graduate Texts in Mathematics Vol.~61, Springer-Verlag, New York-Berlin, 1978.
  
\bibitem{Wi}
J.~Witzig, \emph{Abstract excision and $\ell^1$-homology.} Preprint (2022). arXiv: \url{https://arxiv.org/abs/2203.06120}
\end{thebibliography}
\end{document}